\documentclass[12pt]{article}
\usepackage[english]{babel}
\usepackage{graphicx}
\usepackage{amsmath,amsthm,amssymb}
\usepackage{times}
\usepackage{enumerate}
\usepackage{amsmath}
\usepackage{amsfonts}
\usepackage{pdfsync}
\usepackage{cleveref}
\usepackage{autonum}
\usepackage{enumitem}
\usepackage[usenames, dvipsnames]{color}

\newcommand{\R}{{\mathbb R}}
\newcommand{\N}{{\mathbb N}}
\renewcommand{\L}{\Lambda}

\DeclareMathOperator{\Dom}{Dom}

\DeclareMathOperator{\Lip}{Lip}
\def\titlerunning#1{\gdef\titrun{#1}}
\makeatletter
\def\author#1{\gdef\autrun{\def\and{\unskip, }#1}\gdef\@author{#1}}
\def\address#1{{\def\and{\\}\renewcommand{\thefootnote}{}%
\footnote {#1}}%
\markboth{\autrun}{\titrun}}
\makeatother
\def\email#1{e-mail: #1}
\def\subjclass#1{{\renewcommand{\thefootnote}{}%
\footnote{\emph{Mathematics Subject Classification (2010):} #1}}}
\def\keywords#1{\par\medskip
\noindent\textbf{Keywords.} #1}

\theoremstyle{definition}
\newtheorem{theorem}{Theorem}[section]
\newtheorem{corollary}[theorem]{Corollary}
\newtheorem{lemma}[theorem]{Lemma}

\newtheorem*{theorem*}{Theorem}
\usepackage{amsthm}

\theoremstyle{definition}
\newtheorem{definition}[theorem]{Definition}
\newtheorem{remark}[theorem]{Remark}
\newtheorem{example}[theorem]{Example}

\newtheorem*{hypothesisH1'}{Hypothesis $\mathbf{(H_2')}$}
\newtheorem*{hypothesisH1''}{Hypothesis ($\mathcal S^{\infty}_{x_*}$)}
\newtheorem*{hypothesisS'}{Hypothesis ($\mathcal S^{\infty}_{x_*}$)}
\newtheorem*{hypothesisH'}{Hypothesis ($\mathcal A^{\infty}_{x_*}$)}

\newtheorem*{problemP'}{Problem \textbf{(P$'$)}}

\newtheorem*{conditionS}{Condition (S)}

\newtheorem*{basicass}{Basic Assumptions}

\numberwithin{equation}{section}

\begin{document}
\titlerunning{Non-occurrence of gap}
\title{Non-occurrence of gap for one-dimensional non autonomous functionals}
\author{
 Carlo Mariconda \&
}
\maketitle
\date
\address{
Carlo Mariconda (corresponding author), ORCID $0000-0002-8215-9394$),  Universit\`a
degli Studi di Padova,  Dipartimento di Matematica  ``Tullio Levi-Civita'',   Via Trieste 63, 35121 Padova, Italy;  \email{carlo.mariconda@unipd.it}}
\begin{abstract}
Let $F(y):=\displaystyle\int_t^TL(s, y(s), y'(s))\,ds$ be a positive functional, unnecessarily autonomous,  defined on the space $W^{1,p}([t,T]; \R^n)$ ($p\ge 1$)  of Sobolev functions, possibly with prescribed one or two end point conditions.
It is important, especially for the applications,  to be able to approximate the infimum of $F$ with the values of $F$ along a sequence of Lipschitz functions satisfying the same boundary condition(s). Sometimes this is not possible, i.e., the so called Lavrentiev phenomenon occurs. This is the case of the innocent like Manià's Lagrangian $L(s,y,y')=(y^3-s)^2(y')^6$ and boundary data $y(0)=0, y(1)=1$;  nevertheless in this situation the gap does not occur with just the end point condition $y(1)=1$. The paper focuses about the different set of conditions needed to avoid the gap for problems with just one or with both end point conditions.
Under minimal assumptions on the, possibly extended valued, Lagrangian we ensure the non-occurrence of the Lavrentiev phenomenon with just one end point condition.
We introduce an additional hypothesis, satisfied when the Lagrangian is bounded on bounded sets,  in order to ensure the validity of both end point conditions $y_h(t)=y(t), y_h(T)=y(T)$; the result gives some new light even in the autonomous case.
\end{abstract}
\section{Introduction}
\subsubsection*{The state of the art}
We consider here a one-dimensional, vectorial functional of the calculus of variations
\[F(y)=\int_t^TL(s, y(s), y'(s))\,ds\]
defined on the space of Sobolev functions $W^{1,p}(I, \R^n)$ on $I:=[t,T]$ with values in $\R^n$, for some $p\ge 1$. In the paper the {\em Lagrangian} $L(s,y,u)$ is Borel and is assumed to have values in $[0, +\infty[\cup\{+\infty\}$. Following the terminology of control theory we will refer to $s$ as to the {\em time} variable,  to $y$ as to the {\em state} variable and to $u$ as to the {\em velocity} variable.
The space of absolutely continuous functions provides the correct framework in order to find a minimizer of $F$: Tonelli's Theorem ensures its existence for any given boundary data when $L(s, \cdot, \cdot)$ is lower semicontinuous, $L$ has a superlinear growth and $L(s,y,\cdot)$ is convex. In order to approximate the value of the infimum of $F$ (e.g., by means of numerical methods), one may be tempted to approximate in the space of absolutely continuous functions, a given minimizer $y\in W^{1,p}(I; \R^n)$ with a  sequence of Lipschitz functions $(y_h)_h$ sharing the same boundary data of $y$, in such a way that $\displaystyle\lim_h F(y_h)=F(y)$. It turns out however, that this may not be possible: we say that in this case that the {\em Lavrentiev phenomenon} occurs. This is the case, for instance, for the innocent-looking Manià's \cite{Mania} problem
\[\min F(y):=\int_0^1(y^3-s)^2(y')^6\,ds:\, y\in W^{1,1}(I),\, y(0)=0,\, y(1)=1\]
described in Example~\ref{ex:Mania1}: clearly the minimum is obtained for $y(s):=s^{1/3}$, however it can be shown that there is $\varepsilon>0$ such that $F(z)>\varepsilon$ whenever $z:[0,1]\to\R$ is Lipschitz  and $z(0)=y(0)=0, z(1)=y(1)=1$.
As it is pointed out in \cite{GBH}, it is interesting to note that things change drastically if one allows the approximating sequence to have a different initial datum: indeed the sequence
\[y_h(s):=\begin{cases}{1}/{h^{1/3}}&\text{ if } s\in [0, 1/h],\\
t^{1/3}&\text{ otherwise}, \end{cases}\]
converges to $y$ and $F(y_h)\to F(y)$.
Another celebrated example, due to  Ball and  Mizel,  is provided with a polynomial Lagrangian in $(s,y,u)$ that satisfies Tonelli's existence conditions.

 When a minimizer exists, a way to exclude the Lavrentiev phenomenon is to provide conditions that ensure the {\em Lipschitz continuity of the minimizer} itself. In the  {\em autonomous} case it turns out, starting from the work \cite{CVTrans}  of Clarke and Vinter, up to its refinement  by  Dal Maso and Frankowska \cite{DMF},  that the hypotheses of Tonelli's existence theorem (even without convexity) provide Lipschitzianity of the minimizers. Actually weaker growth conditions than superlinearity ensure both existence and Lipschitzianity: we refer to the pioneer work \cite{Clarke1993} of F. Clarke and to the subsequent papers of  Cellina and his school (see \cite{CTZ, Cellina, MTLip}).
In the {\em non autonomous} case, there are functionals satisfying Tonelli's condition whose minimizers are not Lipschitz or worse, that exhibit the {\em Lavrentiev phenomenon}. {\em Condition }(S), an additional local Lipschitz continuity assumption on the first ``time'' variable of $\L$ (thus always satisfied in the autonomous case), formulated in \S~\ref{sect:SS}, is enough to ensure Lipschitzianity: this fact was established by Clarke in \cite{Clarke1993} and generalized by Bettiol and Mariconda in \cite{BM2}.  Property (S) is known to be a sufficient condition  for the validity of the Du Bois-Reymond equation; we refer to \cite{Cesari} for the smooth case,  to \cite{Clarke1993} for the nonsmooth convex case under weak growth assumptions,  to \cite{BM1, BM2} by Bettiol and the author in the general case.

Many  functionals arising from applications do not fulfill, however,  known existence criteria and the previous regularity approach may thus not be pursued. In this case, looking for the {\em non-occurrence of the Lavrentiev phenomenon} is  even more challenging and interesting. This is easy if one requires some ``natural growth assumptions'' from above and continuous Lagrangians; our interest is therefore devoted to Lagrangians that possibly violate these assumption.\\
Here again, the {\em autonomous case} stands on its own: Alberti and  Serra Cassano in  \cite[Theorem 2.4]{ASC} state that, if $L(s,y,u)=\L(y,u)$ is just {\em Borel} and satisfies
\begin{equation}
\tag{${\rm B}_{\L}$}\forall K>0\,\,\,\exists \nu_0>0\qquad \L\text{ is bounded  on } B_K\times B_{\nu_0}
\end{equation}
then, given $y\in W^{1,p}(I,\R^n)$ such that $\L(y,y')\in L^1(I;\R^n)$ there is  a sequence $(y_h)_h$ of Lipschitz functions converging to $y$ in $W^{1,p}(I;\R^n)$ and such that $I(y_h)\to I(y)$ and $y_h(t)=y(t)$ for all $h$: we say in this case that there is no {\em Lavrentiev gap at $y$} (here and below $B_r$ denotes the closed ball of center the origin and radius $r$ in $\R^n$) for the initial prescribed datum. The violation of Assumption (B$_{\L}$) may lead to the occurrence of the Lavrentiev phenomenon, as shown in Example~\ref{ex:alberti2}.
Though \cite[Remark 2.8]{ASC} conjectures that the result does still hold for the variational problem with {\em both}  end point conditions, the proof of \cite[Theorem 2.4]{ASC} actually holds  for just {\em one} end point condition: Example~\ref{ex:alberti} (Alberti, personal communication), shows that this is not just a  technical issue.  The importance of the boundary datum, and the difficulty of preserving it, was noticed exploring the recent literature also in  the multidimensional case by Bousquet, Treu and Mariconda in \cite{BMT, MT2020, MT2020Open}.\\
In the {\em non autonomous} case the examples (see Manià's) show that some additional conditions have to be added: To the author's knowledge, starting from Lavrentiev himself in \cite{Lavrentiev}, most of the  criteria for the avoidance of the Lavrentiev phenomenon  require that $\L$ is locally Lipschitz or H\"older continuous in the state variable (see \cite{Loewen, TZ, Zas}). As was pointed out by Carlson in  \cite{Carl}, many of them can actually    be obtained as a consequence of   a property introduced by L. Cesari and T.S. Angell in \cite{CeAng}. There are however few exceptions:  the non-occurrence of the Lavrentiev phenomenon was established for {\em two} end point conditions without the above regularity hypotheses on the state variable by:
\begin{itemize}[leftmargin=*]
\item Cellina, Ferriero and Marchini in \cite{CFM} for a class of {\em real valued} Lagrangians of the form $L(s,y,u)=\L(y,u)\Psi(s,y)$ assuming the {\em continuity} of $\L(y,u)$, $\Psi(s,y)>m>0$ (thus excluding Manià's Lagrangian) and the convexity of $u\mapsto \L(y,u)$.
    Notice that the Lagrangian in Manià's example is the product of $\L(y,u)=u^6$ with $\Psi(s,y)=(y-s^3)^2$, but $\Psi$ takes the value 0 (along the minimizer).
\item Mariconda in \cite{MTrans} for a general class of  Lagrangians $\L(s,y,u)$, assuming the local Lipschitz Condition (S) on the time variable $s$, radial convexity  in the velocity variable and a linear growth from below in the velocity variable. It is important to notice that Condition (S) for $\L$ is not an option: the lack of its validity  may lead to the Lavrentiev phenomenon, as in Ball -- Mizel example. In the extended valued case it is assumed, moreover, that $\L$ tends with some uniformity to $+\infty$ as the {\em distance} to the boundary of the effective domain $\Dom(\L)$ (i.e., the set where $\L$ is finite) tends to 0. As in \cite{ASC} the lack of regularity of the Lagrangian is compensated by some local boundedness condition, including (${\rm B}_{\L}$).
\end{itemize}
\subsubsection*{The main results}
These were the main motivations for this paper:
\begin{enumerate}
\item Is it true, as suggested by \cite[Remark 2.8]{ASC}, that the Lavrentiev phenomenon  for the {\em two end point} conditions problem does not occur in the case of  a real valued  autonomous Lagrangian that is bounded on bounded sets?
\item Find a set of assumptions, possibly smaller than those presented in \cite{MTrans} or \cite{CFM}, that guarantee, even in the {\em autonomous} case, the non-occurrence of the phenomenon for the {\em two end points} conditions problem.
\item Find sufficient conditions for the non-occurrence of the Lavrentiev phenomenon with just {\em one} end point condition for {\em non autonomous} Lagrangians.
\end{enumerate}
The main results give actually some light on the problems, and apply to the wider  class of Lagrangians of the form
\begin{equation}\label{tag:Lag}L(s,y,u)=\L(s,y,u)\Psi(s,y),\quad \L, \Psi\ge 0,\end{equation}
with different sets of hypotheses for $\L$ and for $\Psi$. All of the results in the paper assume Condition (S) (just) on $s\mapsto \L(s,y,u)$,  and the continuity of $y\mapsto\Psi(s,y)$. Notice that this class il strictly larger than the one of functions of the form $L(s,y,u)=\L(s,y,u)$ satisfying (S). For instance $L(s,y,u):=\L(s,y,u)\Psi(s,y)$ with
\[\forall s\in [0,1], \forall y,u\in\R\qquad \L(s,y,u):=u^2,\, \Psi(s,y)=\sqrt s\]
does not satisfy (S), but $\L$ does.
 In Section~\ref{sect:PX}  extend the result of Alberti -- Serra Cassano \cite[Theorem 2.4]{ASC} to the wider class of Lagrangians \eqref{tag:Lag}, for the variational problem with  (just) {\em one end}  point constraint
    \begin{equation}\tag{$\mathcal P_X$}\min \{F(y):\, y\in W^{1,p}(I;\R^n),\, y(t)=X\in\R^n\}.\end{equation}
Corollary~\ref{coro:Lav1} shows the non-occurrence of the Lavrentiev phenomenon once $\L$, other than Condition (B$_{\L}$) (obviously adapted to the non autonomous case), satisfies Condition (S). Regarding $\Psi$, it is enough that, for all $K>0$,
\begin{itemize}
 \item[(${\rm B}_{\Psi}$)]  $\Psi$ is bounded on $I\times B_K$;
 \item[(${\rm C}_{\Psi}$)] $\Psi(\cdot, z)$ is continuous for every $z\in B_K$.%
 \item[(${\rm P}_{\Psi}$)] $\inf\Psi>0$.
 \end{itemize}
The  above gives an answer to Question 3. Concerning Question 2, we show that, give $y\in W^{1,p}(I;\R^n)$ such that $F(y)<+\infty$, there is no Lavrentiev gap for the {\em two end point} conditions problem once one assumes, moreover, that:
  \begin{itemize}
\item[(${\rm U}_{y,\L}$)] There is an open subset $U_y$ of $y(I)$ such that, for all $r>0$, $\L$ is bounded on $I\times U_y\times B_r$.
\end{itemize}
Hypothesis  (${\rm U}_{y,\L}$) was conjectured by Alberti in a personal communication for the autonomous case. As shown in Example~\ref{ex:alberti}, its lack may lead to a gap. Notice that Hypothesis (${\rm U}_{y,\L}$) subsumes the fact that $\L$ is real valued on an infinite strip. Thus,  in order to be satisfied for any admissible trajectory $y$, we require in Corollary~\ref{coro:Lav1}  that $\L$ is bounded on bounded sets (and thus a fortiori,  that $\L$ is not extended valued).   This gives, in any case, a reassuring positive answer to Question 1, confirming the statement in \cite[Remark 2.8]{ASC}.

When $\Psi\equiv 1$,  the conclusions of Theorem~\ref{thm:Lav1} and  Corollary~\ref{coro:Lav1} do  not overlap with those obtained in \cite[Corollary 6.7]{MTrans} for the same kind of Lagrangians. Some extra hypotheses assumed there, e.g., the  radial convexity assumption on $0<r\mapsto \L(s,z,rv)$, reveals to be more suitable for extended valued Lagrangians and allow to deal with state constraints, out or reach with the methods of the present paper.
\subsubsection*{A word on the proof}
The proof of Theorem~\ref{thm:Lav1} follows the path of that of \cite[Theorem 2.4]{ASC}. Given  $y\in W^{1,p}(I;\R^n)$ such that $F(y)<+\infty$, we build  a sequence $(y_h)_h$ of Lipschitz functions, satisfying the required boundary conditions/state constraints, and:
\begin{enumerate}
 \item $\displaystyle\lim_{h\to +\infty}F(y_h)=F(y)$;
 \item $y_h\to y$ in $W^{1,p}(I;\R^n)$.
\end{enumerate}
We begin by considering a standard approximating sequence $(z_h)_h$ of Lipschitz functions such that $z_h'=y'$ everywhere except at most an open subset $A_h$ whose measure tends to 0 as $m\to +\infty$. We then reparametrize  each $z_h$, i.e., we set $y_h:=z_h\circ\psi_h$ for a suitable Lipschitz injective function  $\psi_h:I\to I$. A delicate point here, is that in general (B$_{\L}$) and the assumptions for $\Psi$ alone do not ensure that $\psi_h(I)=I$. This is the point where, for the two end point conditions problem, Hypothesis (U$_{y, \L}$) plays a role, allowing to build  a suitable change of variables  $\varphi_{h}$, satisfying  $\varphi_h(t)=t$ and  $\varphi_h=T$ in such a way that $y_h:=z_h\circ\psi_h$ fulfills the required properties.
\section{Notation and Basic assumptions}
\subsection{Basic Assumptions}
Let $p\ge 1$. The functional $F$ (sometimes referred as to the ``energy'')  is defined by
\[\forall y\in W^{1,p}(I,\R^n)\qquad F(y):=\int_IL(s, y(s), y'(s))\, ds,\]
where $L(s,y,v)$ is of the form $L(s,y,v)=\L(s, y, v)\Psi(s,y)$.
\begin{basicass}
We assume  the following conditions.
\begin{itemize}
\item $I=[t, T]$ is a closed, bounded  interval of $\R$;
\item $\L:I\times \R^n\times\R^n\to [0, +\infty[\cup\{+\infty\},\, (s,y,u )\mapsto \L(s,y,u )$ ($n\ge 1$) is Borel measurable;
\item $\Psi:I\times\R^n\to [0, +\infty]$ is  Borel.
\item    The  {\bf effective domain} of $\L$, given by
   \[\Dom (\L):=\{(s,y,u)\in I\times\R^n\times\R^n:\, \L(s,y,u )<+\infty\}\]
   is of the form $\Dom(\L)=I\times D_{\L}$, with $D_{\L}\subseteq\R^n\times\R^n$.
\end{itemize}
\end{basicass}
\subsection{Notation}
We introduce the main recurring notation:
\begin{itemize}
\item The Euclidean norm of  $x\in \R^n$ is denoted by $|x|$;
\item The Lebesgue measure of a subset $A$ of $I=[t,T]$ is $|A|$ (no confusion may occur with the Euclidean norm);
\item If $y:I\to\R^n$ we denote by $y(I)$ its image, by $\|y\|_{\infty}$ its sup-norm and by $\|y\|_p$ its norm in $L^p(I;\R^n)$;
\item The complement of a set $A$ in $\R^n$ is denoted by $A^c$;
\item The characteristic function of a set $A$ is $\chi_A$.
\item If $x\in \R$, we denote by $x^+$ its positive part, by $x^-$ its negative part;
\item  $\Lip(I;\R^n)=\{y:I\to\R^n,\, y\text{ Lipschitz}\}$;  if $n=1$ we simply write $\Lip(I)$;
\item For $p\ge 1$, $W^{1,p}(I;\R^n)=\{y:I\to\R^n:\, y, y'\in L^p(I;\R^n)\}$; if $n=1$ we simply write $W^{1,p}(I)$.
\end{itemize}
\subsection{Two variational problems}
We shall consider  different variational problems associated to the functional $F$, with different end-point conditions and/or state constraints.
Let $X, Y\in \R^n$. We define
\begin{itemize}
\item $\Gamma_X:=\{y\in W^{1,p}(I,\R^n):\, y(t)=X\}$, and the corresponding variational problem
\begin{equation}\tag{$\mathcal P_X$}\text{Minimize }\{F(y):\, y\in \Gamma_X\},\end{equation}
whenever $\inf (\mathcal P_{X})<+\infty$.
\item $\Gamma_{X, Y}:=\{y\in W^{1,p}(I,\R^n):\, y(t)=X,\, y(T)=Y\}$, and the corresponding variational problem
\begin{equation}\tag{$\mathcal P_{X, Y}$}\text{Minimize }\{F(y):\, y\in \Gamma_{X, Y}\},\end{equation}
whenever $\inf (\mathcal P_{X,Y})<+\infty$.
\end{itemize}
There is no privilege in considering the initial condition $y(t)=X$ instead of the final one $y(T)=Y$ for the one end point conditions problem: any result obtained here  can be reformulated for  a final end point condition variational problem, with the same set of assumptions.
\subsection{Lavrentiev gap at a function and Lavrentiev phenomenon}
In this paper we consider different boundary data for the same integral functional.
\begin{definition}[\textbf{Lavrentiev gap at $y\in W^{1,p}(I;\R^n)$}]
Let $y\in W^{1,p}(I;\R^n)$ be such that $F(y)<+\infty$ and let $\Gamma\in \{\Gamma_X, \Gamma_{X,Y}\}$.\\ We say that the \textbf{Lavrentiev gap} does not occur at $y$ for the variational problem corresponding to $\Gamma$ if
there exists a sequence $\left(y_h\right)_{h\in\N}$ of functions in $\Lip (I, \R^n)$   satisfying:
 \begin{enumerate}
 \item $\forall h\in\mathbb N\quad y_h\in\Gamma$;
 \item $\displaystyle\limsup_{h\to +\infty}F(y_h)= F(y)$;
 \item $y_h\to y$ in $W^{1,p}(I;\R^n)$.
\end{enumerate}
We say that the \textbf{Lavrentiev phenomenon} does not occur for the variational problem corresponding to $\Gamma$ if
\begin{equation}\label{tag:Lavrentievdefi}
\inf_{\substack{y\in W^{1,p}(I; \R^n)\\y\in \Gamma}}F(y)=\inf_{\substack{
y\in \Lip(I; \R^n)\\
y\in \Gamma}}F(y).
\end{equation}
\end{definition}
\begin{remark}[Gap and phenomenon] \phantom{AA}
\begin{itemize}
\item Let $y\in W^{1,p}(I;\R^n)$. The non-occurrence of the phenomenon at $y$ ensures that, given $\varepsilon>0$, there is a {\em Lipschitz} function $\overline y$ satisfying the same boundary data and/or constraints such that $F(\overline y)\le F(y)+\varepsilon$. If $L(s,y,\cdot)$ is convex for all $(s,y)\in I\times\R^n$, and $L(s, \cdot, \cdot)$ is lower semicontinuous, the non-occurrence of the Lavrentiev gap at $y$ implies the convergence of $(y_h)_h$ to $y$ in {\em energy}, i.e.,
    \[\lim_{h\to +\infty}F(y_h)=F(y):\] Indeed in that case  $F$ is weakly lower semicontinuous.
\item Of course, the non-occurrence of the Lavrentiev gap  along a minimizing sequence
implies the non-occurrence of the Lavrentiev phenomenon for the same variational problem.
\end{itemize}
\end{remark}
The following celebrated example motivates the need to distinguish problems with just {\em one} end point condition from problems with  {\em both} end points conditions.
\begin{example}[Manià's example \cite{Mania}]\label{ex:Mania1} Consider the problem of minimizing
\begin{equation}\tag{$\mathcal P_{0,1}$}F(y)=\int_0^1(y^3-s)^2(y')^6\,ds:\, y\in W^{1,1}(I),\, y(0)=0,\, y(1)=1.\end{equation}
Then $y(s):=s^{1/3}$ is  a minimizer and $F(y)=0$. Not only $y$ is not Lipschitz; it turns out (see \cite[\S 4.3]{GBH}) that the Lavrentiev phenomenon occurs, i.e.,
\[0=\min F=F(y)<\inf\{F(y):\, y\in \Lip([0,1]),\, y(0)=0, y(1)=1\}.
\]
However, as it is noticed in \cite{GBH}, the situation changes drastically if one allows to vary the initial boundary condition along the sequence $(y_h)_h$. Indeed it turns out that the sequence $(y_h)_h$, where each $y_h$ is obtained by truncating $y$ at $1/h$, $h\in\mathbb N_{\ge 1}$, as follows:
\[y_h(s):=\begin{cases}{1}/{h^{1/3}}&\text{ if } s\in [0, 1/h],\\
s^{1/3}&\text{ otherwise}, \end{cases}\]
\begin{figure}[h!]
\begin{center}
\includegraphics[width=0.4\textwidth]{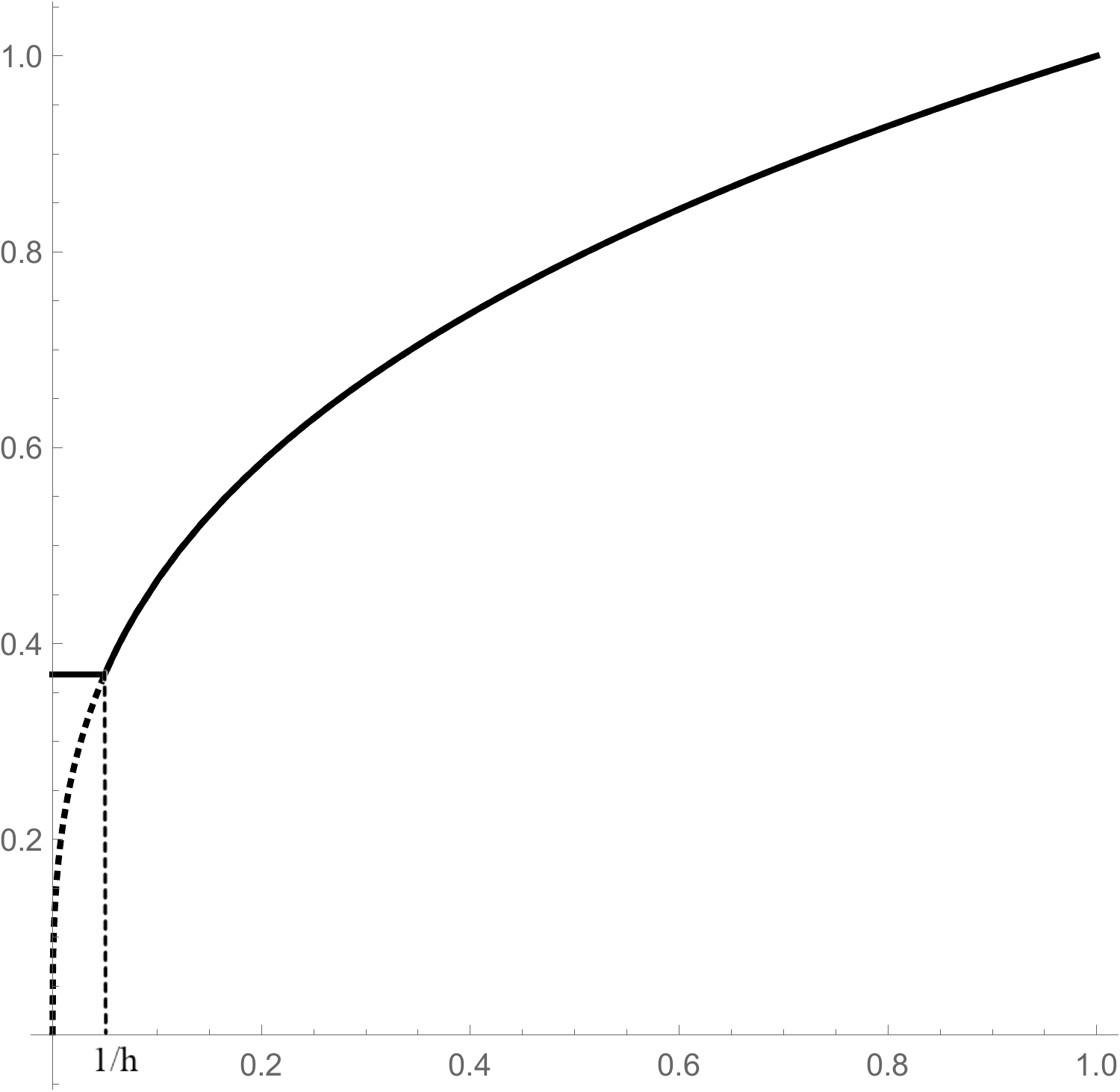}
\caption{{\small The function $y_h$.}}
\end{center}
\end{figure}
is a sequence of Lipschitz functions satisfying
\[y_h(1)=y(1)=1,\quad F(y_h)\to F(y),\quad y_h\to y \text{ in }W^{1,1}([0,1]).\]
Therefore, no Lavrentiev phenomenon occurs for the variational problem
\begin{equation}\min F(y)=\int_0^1(y^3-s)^2(y')^6\,ds:\, y\in W^{1,1}(I),\, \, y(1)=1.\end{equation}
\end{example}
\subsection{Condition (S)}\label{sect:SS}
We consider the following local Lipschitz condition (S) on the first variable of $\L$.
\begin{conditionS}
For every  $K\ge 0$ of $\R^n$ there are $\kappa,  \beta\ge 0, \gamma\in L^1[t,T]$,
$\varepsilon_*>0$ satisfying,
for a.e.  $s\in I$
 \begin{equation}\label{tag:H3}
|\L(s_2,z,v)-\L(s_1,z,v)|\le \big(\kappa \L(s,z,v) +\beta|v|^p+\gamma(s)\big)\,|s_2-s_1|
\end{equation}
whenever $s_1,s_2\in [s-\varepsilon_*,s+\varepsilon_*]\cap I$, $z\in B_K$, $v \in\R^n$, $(s,z,v) \in\Dom(\L)$.
\end{conditionS}
\begin{remark} Condition (S) is fulfilled if $\L=\L(y,u)$ is autonomous. In the smooth setting, Condition (S) ensures the validity of the Erdmann - Du Bois-Reymond (EDBR) condition. In this more general framework it plays a key role in Lipschitz regularity under slow growth conditions \cite{Clarke1993, BM2, MTrans}  and ensures the validity of the (EDBR) for real valued Lagrangians \cite{BM1, BM2}.
\end{remark}
\subsection{A useful option: linear growth from below for $\L$}
The following additional linear growth from below on $\L$ is not assumed in the main results; however its validity allows to weaken some of the hypotheses of Corollary~\ref{coro:Lav1} below.
\begin{itemize}
\item  There are $\alpha>0$ and $d\ge 0$ satisfying, for a.e. $s\in[0,T]$ and every $z\in \R^n , v\in \R^n$,
\begin{equation}\label{tag:lingrowth}\tag{{\rm G}$_{\L}$}\L(s,z,{v})\ge \alpha|{v}|-d.\end{equation}
\end{itemize}
\begin{lemma}\label{lemma:linbelow}
Let $y\in W^{1,p}(I; \R^n)$ be such that $F(y)<+\infty$. Assume  that $\L$ fulfills ({\rm G}$_{\L}$) and that the infimum of $\Psi$ along the graph of $y$ is strictly positive, i.e.,
\begin{itemize}
\item[(${\rm P}_{y,\Psi}$)] There is $m_{y, \Psi}>0$ such that $\Psi(s, z)\ge m_{y, \Psi}$ for all $s\in I, z\in y(I)$.
\end{itemize}
Then
\[\|y\|_1\le \dfrac{F(y)+\displaystyle m_{y,\Psi}d(T-t)}
{\displaystyle m_{y,\Psi}\alpha}.\]
\end{lemma}
\begin{proof} From (${\rm P}_{\Psi}$) and ({\rm G}$_{\L}$) we obtain
\[\begin{aligned}F(y)&=\int_t^T\L(s,y(s), y'(s))\Psi(s,y(s))\,ds\\
&\ge m_{y,\Psi}\,\int_t^T\L(s, y(s), y'(s))\,ds\\
&\ge  m_{y,\Psi}\alpha\int_t^T|y'(s)|\,ds-
m_{y,\Psi}d(T-t),\end{aligned}\]
proving that
\[\int_t^T|y'(s)|\,ds\le \dfrac{F(y)+\displaystyle m_{y,\Psi}d(T-t)}
{\displaystyle m_{y,\Psi}\alpha}.\]
\end{proof}
%
\section{Non-occurrence of the Lavrentiev gap/phenomenon for ($\mathcal P_X$) and for ($\mathcal P_{X, Y}$) }\label{sect:PX}
Let $X\in\R^n$. We consider here  problems ($\mathcal P_X$) ($\mathcal P_{X, Y}$).
\subsection{Nonoccurrence of the Lavrentiev gap}
Theorem~\ref{thm:Lav1} below extends \cite[Theorem 2.4]{ASC} to nonautonomous Lagrangians: indeed if $\L(s,y,u)=\L(y,u)$ then Condition (S) is fulfilled.
\begin{theorem}[\textbf{Non-occurrence of the Lavrentiev  gap for {\rm(}$\mathcal P_X${\rm)}  and for {\rm(}$\mathcal P_{X, Y}${\rm)} at $y\in W^{1,p}(I;\R^n)$}]\label{thm:Lav1} Assume that $\L$ satisfies Condition (S)  and let $y\in W^{1,p}([t, T],\R^n)$ be such that \[\L(s, y(s), y'(s))\Psi(s, y(s))\in L^1([t, T]).\]
Moreover, suppose that there is a neighbourhood $\mathcal O_y$ of $y(I)$ in $\R^n$ such that:
\begin{itemize}
\item[(${\rm B}_{y,\Psi}$)] $\Psi$ is bounded on $I\times \mathcal O_y$;
\item[(${\rm C}_{y,\Psi}$)]$\Psi(\cdot, z)$ is continuous for all $z\in y(I)$;
\item[(${\rm B}_{y,\Lambda}$)] There is   $\nu_0>0$ such that $\L$ is bounded on $I\times \mathcal O_y\times B_{\nu_0}$.
\end{itemize}
Moreover, assume that  $\L(s, y, y')\in L^1(I)$.
Then:
\begin{enumerate}
\item  There is {no Lavrentiev gap} for ($\mathcal P_X$) at $y$.
\item Assuming, moreover, that:
\begin{itemize}[leftmargin=*]
\item[(${\rm U}_{y,\L}$)] There is an open subset $U_y$ of $y(I)$ such that, for all $r>0$, $\L$ is bounded on $I\times U_y\times B_r$,
\end{itemize}
then there is {no Lavrentiev gap} for ($\mathcal P_{X, Y}$) at $y$.
\end{enumerate}
\end{theorem}
\begin{remark}\label{rem:positivity}Notice that in Theorem~\ref{thm:Lav1}, the integrability of $\L(s, y, y')$ is satisfied if, for instance, $\Psi$ satisfies Condition (P$_{y, \Psi}$) formulated in Lemma~\ref{lemma:linbelow}. Indeed, if $\Psi\ge m_{y, \Psi}$ on $I\times y(I)$, then
\[\int_t^T\L(s, y(s), y'(s))\,ds\le \dfrac1{m_{y,\Psi}}F(y)<+\infty.\]
\end{remark}
Manià example~\ref{ex:Mania1} shows that the integrability of $\L(s,y,y')$ is not a technical matter for the problem with two and, as shown below,  even one end point condition.
 \begin{example} Consider Manià's Example~\ref{ex:Mania1}. The Lagrangian may be written as
\[L(s,y,u)=(y^3-s)^2u^6=\L(u)\Psi(s,y),\quad \Psi(s,y)=(y^3-s)^2,\,\L(u)=u^6.\]
Consider the minimizer $y(s):=s^{1/3}$.
All of  the assumptions of Theorem~\ref{thm:Lav1} are satisfied, except the integrability of $\L(y')$; notice also that $\Psi$ vanishes along $(s, y(s))$ (see Remark~\ref{rem:positivity}).
As a matter of fact, the Lavrentiev gap occurs at $y$, for the variational problem with (just) initial prescribed datum
\[\min F(z):=\int_0^1(z^3-s)^2(z')^6\,ds,\quad z(0)=0.\]
Indeed, if $y_h$ is a sequence of Lipschitz functions that converges to $y$ in $W^{1,1}$ then $y_h(1)$ is definitely in any given neighborhood of $y(1)=1$, say in $[3/4, 3/2]$: the arguments of \cite[\S 4.3]{GBH} show that there is $\eta>0$ such that $F(y_h)\ge \eta$ for $h$ big enough.
\end{example}
\subsection{Examples}
The next examples concern the autonomous case, i.e., $\L=\L(z,v)$ and $\Psi\equiv 1$.
Example~\ref{ex:alberti} below shows that Hypothesis (${\rm U}_{y,\L}$) is essential for the validity of Claim 2 in Theorem~\ref{thm:Lav1}, when $\L$ is extended valued. It is a slight  modification of an example
 by G. Alberti (personal communication).
\begin{example}[Occurrence of the Lavrentiev phenomenon in  an autonomous, convex  and l.s.c. problem with both endpoint constraints]\label{ex:alberti}
Let
 $y\in W^{1,1}([0,1];\R)$ be such that
\begin{itemize}
\item $y$ is of class $C^1$ in $[0,1[$, $y(0)=0, y(1)=1$;
\item $ y'>0$ on $[0, 1[$,
\item $y'(1):=\displaystyle\lim_{s\to 1^-}y'(s)=+\infty$.
    \end{itemize}
Such a function exists,  e.g., $y(s):=1-\sqrt{1-s}, s\in [0,1]$.
For every $z\in [0,1[$ set $q(z):=y'(y^{-1}(z))$. Let, for $(s,y,v)\in [0,1]\times \R\times\R$,
\[\L(s,z,v):=\begin{cases} 0&\text{ if } z\in [0,1[\text{ and } v\le q(z)\text{ or } z\notin [0,1[,\\
+\infty&\text{ otherwise},\end{cases}\qquad  \Psi\equiv 1,\]
\begin{figure}[h!]
\begin{center}
\includegraphics[width=0.55\textwidth]{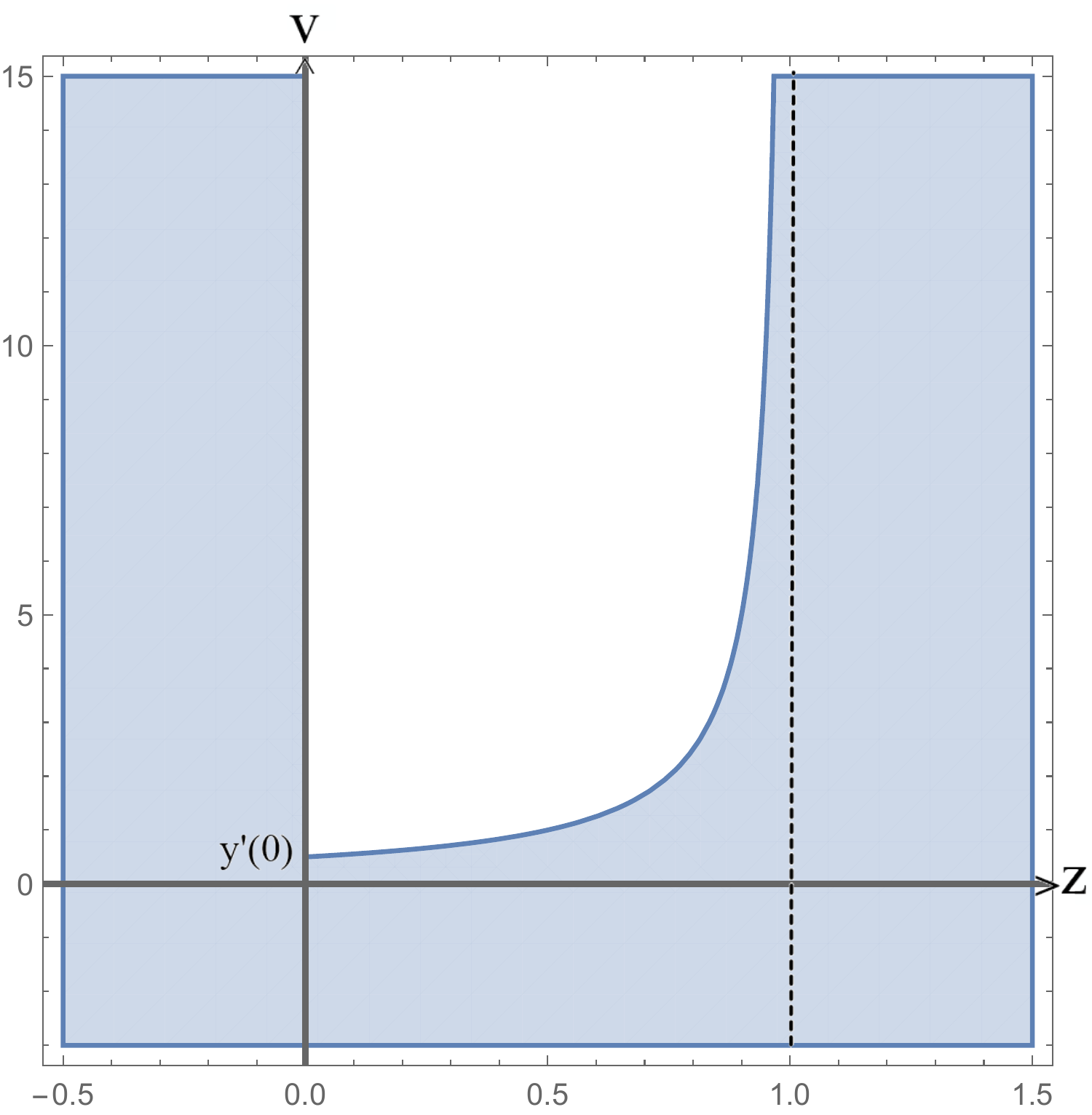}
\caption{{\small The domain of $\L(s, \cdot, \cdot)$ in Example~\ref{ex:alberti}}}
\label{fig:exalberti}\end{center}
\end{figure}
and set $F(z):=\displaystyle\int_0^1\L(s, z(s), z'(s))\,ds$ for every $z\in W^{1,1}([0,1];\R)$.
Clearly $F(y)=\min F=0$. Notice that:
\begin{itemize}
\item $\L$ is autonomous;
\item  $\L$ is lower semicontinuous on $\R^2$ and $\L(s, z, \cdot)$ is convex for all $z\in\R$;
\item $\L$ is bounded on $\R\times \R\times ]-y'(0), y'(0)[$;
\item $\L$ takes the value $+\infty$ on every set of the form $[0,1]\times ]a,b[\times \R$.
\end{itemize}
Thus $\L, \Psi$ satisfy all the assumptions of Claim 1 in Theorem~\ref{thm:Lav1}. In particular, there is no Lavrentiev  gap at $y$ (and thus phenomenon) for the problem with just one end point condition: either $y(0)=0$ or $y(1)=1$. However,
 condition (U$_{y, \L}$) is not fulfilled.
\\
{\em Claim. $F(z)=+\infty$ for every Lipschitz $z:[0,1]\to\R$ satisfying $z(0)=0, z(1)=1$.}
 Indeed assume the contrary: let $z$ be such a function and suppose $F(z)<+\infty$.
Let $0\le t_1<t_2\le 1$ be such that $z(t_1)=0, z(t_2)=1$ and $z([t_1, t_2])=[0, 1]$.
Since $F(z)<+\infty$ then
\begin{equation}\label{tag:q1}z'(s)\le q(z(s))\text{ a.e.  on } [t_1,t_2].\end{equation}
Notice that, since  $\displaystyle\lim_{s\to 1}q(z(s))=+\infty$ and $z'$ is bounded, then necessarily \eqref{tag:q1} is strict on a non negligible set.  It follows that
\[\int_{t_1}^{t_2}\dfrac{z'(s)}{q(z(s))}\,ds<\int_{t_1}^{t_2}\,ds=t_2-t_1\le 1.\]
However the change of variable $\zeta=z(s)$ (which is justified, for instance, by the chain rule \cite[Theorem 1.74]{MalyZiemer}), gives
\[\begin{aligned}\int_{t_1}^{t_2}\dfrac{z'(s)}{q(z(s))}\,ds&=\int_{0}^1\dfrac1{q(\zeta)}\,d\zeta\\
&=\int_0^1\dfrac{1}{y'(y^{-1}(\zeta))}\,d\zeta\\
&= \small{(\tau=y^{-1}(\zeta))} \int_0^1\dfrac{y'(\tau)}{y'(\tau)}\,d\tau=1,
\end{aligned}\]
a contradiction, proving the claim.
\end{example}
The violation of Condition  (${\rm B}_{y,\L}$) in  Theorem~\ref{thm:Lav1} may cause the occurrence of the phenomenon, even  for one point constraint problems.
\begin{example}[Occurrence of the phenomenon in  autonomous scalar problems with  one endpoint constraint]\label{ex:alberti2}
Let $y$ and $q$ be as in Example~\ref{ex:alberti}, assume moreover that $y\in C^2$ and $y''>0$ on $[0,1[$. Define, for $(s, z, v)\in [0,1]\times\R\times\R$,
\[\L(s,z,v):=\begin{cases} 0&\text{ if } z\in I:=[0,1]\text{ and } v\ge q(z),\\
+\infty&\text{ otherwise},\end{cases}\]
\begin{figure}[h!]
\begin{center}
\includegraphics[width=0.45\textwidth]{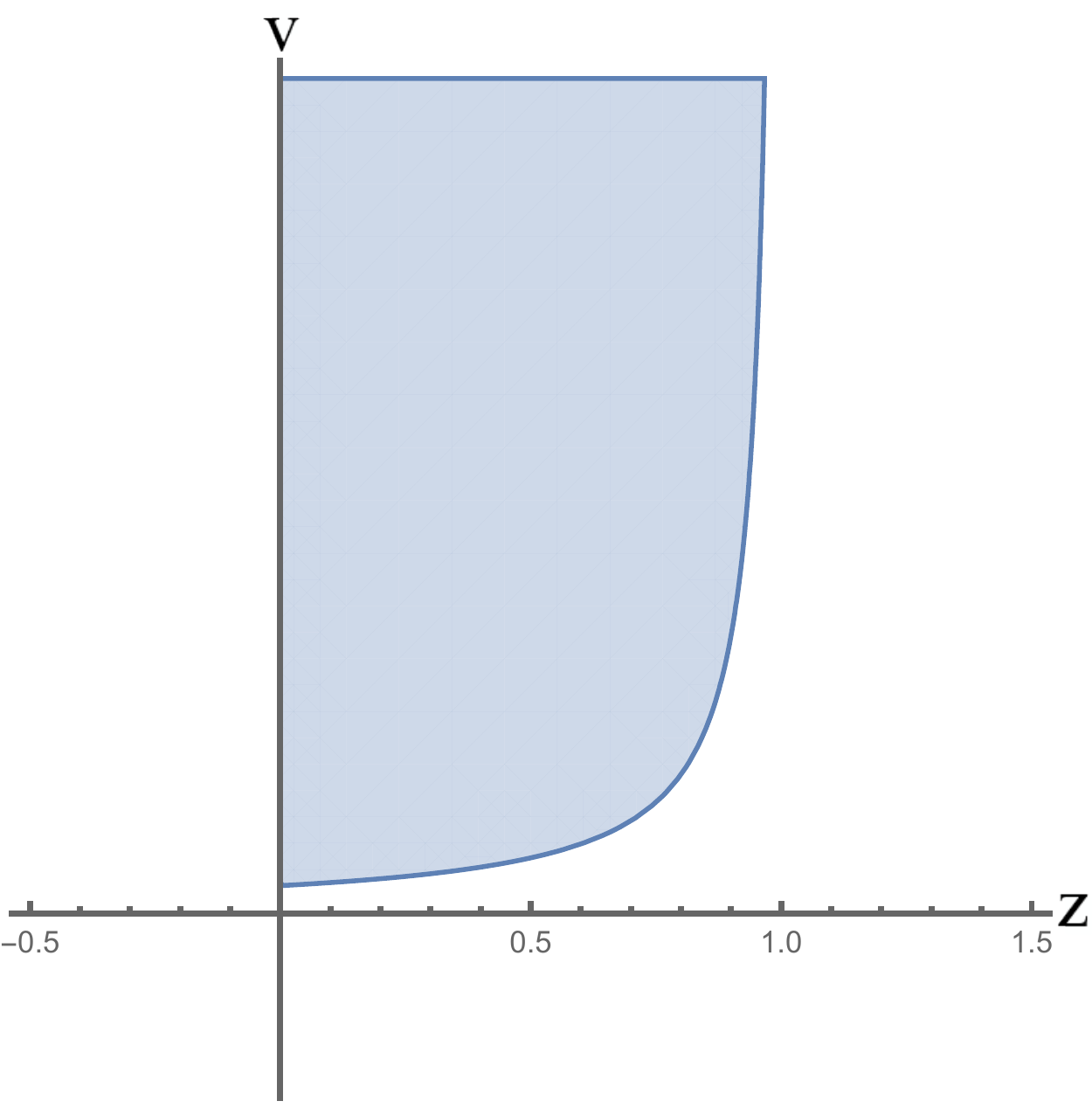}
\caption{{\small The domain of $\L$ in Example~\ref{ex:alberti2}}}
\label{fig:exalberti2}\end{center}
\end{figure}
and set $F(z):=\displaystyle\int_0^1\L(s, z(s), z'(s))\,ds$ for every $z\in W^{1,1}([0,1];\R)$.
Notice that $\L$ is autonomous, lower semicontinuous and $\L(s, z,\cdot)$ is convex for every $z\in\R$. The Lagrangian here is of the form $\L(s,z,z')\Psi(s,z)$ with $\Psi\equiv 1$.  Clearly $\Psi$  satisfies the required assumptions of Theorem~\ref{thm:Lav1} and $\inf\Psi>0$, whereas $\L$ takes the on every rectangle $I\times y(I)\times B_{\nu_0}$ and thus  violates (${\rm B}_{y,\L}$) in  Theorem~\ref{thm:Lav1}.\\
{\em Claim.
The Lavrentiev phenomenon occurs for $F$  with the end point condition $z(1)=1$.}
\\
Clearly $\L(s, y, y')=0$ a.e. in $[0,1]$, so that $F(y)=\min F=0$. However, $F(z)=+\infty$ for every Lipschitz function $z$ satisfying $z(1)=1$.
Indeed, let $z(\cdot)$ be  a Lipschitz function on $[0, 1]$ satisfying $z(1)=1$. Let $M\ge 0$ be such that  $z'\le M$ a.e. on $[0, 1]$ and $t_*\in [0,1[$ be such that $y'(t_*)> M$. Since $y(t_*)<1=z(1)$ there is, by continuity,
 $\varepsilon>0$  such that
 \begin{equation}\label{tag:iqdgdqiugd}\forall s\in [1-\varepsilon, 1]\qquad   y(t_*)<z(s).\end{equation}
Then $\L(s,z(s), z'(s))=+\infty$ on $[1-\varepsilon, 1]$.  Indeed, if $s\in [1-\varepsilon,1]$, the monotonicity of $y'$ implies
\[
z'(s)\le M<y'(t_*)<y'(y^{-1}(z(s))=q(z(s)),
\]
whence $\L(z(s), z'(s))=+\infty$. It follows that $F(z)=+\infty$.
\end{example}
\subsection{Nonoccurrence of the Lavrentiev phenomenon}
We immediately obtain a sufficient condition for the non-occurrence of the Lavrentiev gap. The hypotheses of Corollary~\ref{coro:Lav1} ensure that the conditions of Theorem~\ref{thm:Lav1} are satisfied for every $y\in W^{1,p}(I;\R^n)$ with $F(y)<+\infty$.
\begin{corollary}[\textbf{Non-occurrence of the Lavrentiev phenomenon for {\rm(}$\mathcal P_X${\rm)}}]\label{coro:Lav1}
 Assume the validity of the Basic Assumptions  and, moreover, that, for all $K>0$:
\begin{itemize}
 \item[(${\rm B}_{\Psi}$)]  $\Psi$ is bounded on $I\times B_K$;
  \item[(${\rm C}_{\Psi}$)] $\Psi(\cdot, z)$ is continuous for every $z\in B_K$;%
  \item[(${\rm P}_{\Psi}$)] There is $m_{K, \Psi}>0$ such that $\Psi(s, z)\ge m_{K, \Psi}$ for all $s\in I, z\in B_K$;
\item[(${\rm B}_{\Lambda}$)] There is   $\nu_0>0$ such that $\L$ is bounded on $I\times B_K\times B_{\nu_0}$.
    \end{itemize}
Then:
\begin{enumerate}
\item The {Lavrentiev phenomenon does not occur} for ($\mathcal P_X$).
\item If $\L$ is real valued and,
\begin{itemize}[leftmargin=*]
\item[(U$_{\L}$)]
For all $r>0$, $\L$ is bounded on $I\times B_K\times B_r$,
\end{itemize}
then the
{Lavrentiev phenomenon does not occur} for ($\mathcal P_{X, Y}$).
\item  If, in addition to the assumptions,  $\L$ fulfills ({\rm G}$_{\L}$) and $\inf \Psi=m_{\Psi}>0$, the conclusions of Claims 1, 2 hold whenever  Hypotheses (${\rm B}_{\Psi}$), (${\rm C}_{\Psi}$), (${\rm B}_{\Lambda}$) are satisfied  for just one value of $K>K_0$, where     \begin{equation}\label{tag:K0} K_0:= |X|+\dfrac{\inf {\rm(}\mathcal P_X{\rm)}+\displaystyle m_{\Psi}d(T-t)}
{\displaystyle m_{\Psi}\alpha}.\end{equation}
\end{enumerate}
 \end{corollary}
 \begin{proof} 1. Let $(\overline y_h)_h$ be a minimizing sequence for (P$_X$) such that
\[\forall h\in\mathbb N\qquad F(\overline y_h)\le \inf {\rm(}\mathcal P_X{\rm)}+\dfrac1{h+1}.\]
 Fix $h\in\mathbb N$ and choose $K=K_h>0$ in such a way that $B_{K_h}$ contains a neighborhood of $\overline y_h(I)$; the validity of hypotheses (${\rm B}_{\L}$), (${\rm B}_{\Psi}$), (${\rm C}_{\L}$) implies that of (${\rm B}_{\overline y_h, \L}$), (${\rm B}_{\overline y_h, \Psi}$), (${\rm C}_{\overline y_h, \L}$) with $\mathcal O_h:=B_{K_h}$, and that of (P$_{\overline y_h, \Psi}$). By applying Theorem~\ref{thm:Lav1} we obtain  the existence of $y_h\in\Lip(I;\R^n)$
 satisfying the boundary condition $y_h(t)=X$  and
 \[F(y_h)\le F(\overline  y_h)+\dfrac1{h+1}\le  \inf {\rm(}\mathcal P_X{\rm)}+\dfrac2{h+1}.\]
 The claim follows.\\
 2. Consider a minimizing sequence $(\overline y_h)$ for ($\mathcal P_{X, Y}$) such that
\[\forall h\in\mathbb N\qquad F(\overline y_h)\le \inf {\rm(}\mathcal P_{X, Y}{\rm)}+\dfrac1{h+1}.\]
Notice that, for each $h\in\mathbb N$,  the validity of Hypothesis (U$_{\Lambda}$) implies that of (U$_{\overline  y_h,\Lambda}$). We proceed as in
 Claim 1.
 \\
 3.  If $({\rm P}_{\Psi}$) and  ({\rm G}$_{\L}$) hold, from Lemma~\ref{lemma:linbelow} we may assume in the proof of Claims 1 and 2 that
 \[\begin{aligned}\|\overline  y_h\|_{\infty}&\le |X|+\|\overline y_h\|_1\\
 &\le |X|+\dfrac{F(\overline  y_h)+\displaystyle m_{\Psi}d(T-t)}
{\displaystyle m_{\Psi}\alpha}\\
&\le |X|+\dfrac{\inf {\rm(}\mathcal P_X{\rm)}+\frac1{h+1}+\displaystyle m_{\Psi}d(T-t)}
{\displaystyle m_{\Psi}\alpha}< K,\end{aligned}\]
once $h\ge  h_0$, with $\dfrac1{(h_0+1)m_{\Psi}\alpha}<K-K_0$.
We can thus take,  for {\em all} $h\ge h_0$, $K_h:\equiv K$ in remaining part of the proof. The conclusion follows.
  \end{proof}
\begin{remark}Here are some comments concerning the assumptions  of  Theorem~\ref{thm:Lav1} and Corollary~\ref{coro:Lav1}.
 \begin{itemize}
 \item Hypothesis  (${\rm B}_{\Lambda}$) is required in \cite[Theorem 2.4]{ASC} for autonomous Lagrangians of the form $L(s,y,u)=\L(y,u)$.
 \item The validity of (${\rm B}_{\Psi}$) (resp.  of (${\rm C}_{\Psi}$)) for {\em every} $K>0$  is of course equivalent to the fact that $\Psi$ is bounded on bounded sets (resp.  that $\Psi(\cdot, z)$ is continuous for every $z\in\R^n$). Claim 3 of Corollary~\ref{coro:Lav1} explains the choice of the apparently naive formulation of the assumptions.
 \item The validity of (${\rm B}_{\Lambda}$) for {\em every} $K>0$ implies that $\L$ is bounded in $I\times \R^n\times \{0\}$; as shown in Corollary~\ref{coro:Lav1} this inconvenient is encompassed if$\L$ satisfies ({\rm G}$_{\L}$).
 \item Hypothesis (U$_{y, \L}$) implies that the effective domain of $\L$ contains the unbounded strip $I\times U_y\times \R^n$; Hypothesis (U$_{\L}$) forces $\L$ to be real valued or, if (G$_{\L}$) holds, at least real valued on $I\times B_{K_0}\times \R^n$ where $K_0$ is given by \eqref{tag:K0}.
     \end{itemize}
 \end{remark}
 In the real valued, continuous case, many of the assumptions of Corollary~\ref{coro:Lav1} are satisfied whenever $\L$ si bounded on bounded sets.
 \begin{corollary}[\textbf{Non-occurrence of the Lavrentiev phenomenon for  {\rm(}$\mathcal P_{X}${\rm)} -- real valued case}]\label{coro:Lav1realcont}
 Assume the validity of the Basic Assumptions  and, moreover, that $\L$ is real valued, bounded on bounded sets, and that $\Psi$ is continuous and strictly positive.
Then the Lavrentiev phenomenon does not occur for ($\mathcal P_{X, Y}$).
 \end{corollary}
 \section{Proof of Theorem~\ref{thm:Lav1}}
Lemma~\ref{lemma:2.6} is a slight modification of \cite[Lemma 2.6]{ASC}.
\begin{lemma}\label{lemma:2.6} Let $(g_h)_h$ be a sequence of functions in $L^1(I, [0, +\infty[)$ converging to $g\in L^1(I)$ a.e. in $I$ and $(E_h)_h$ be a sequence of measurable subsets of $I$ such that $|I\setminus E_h|\to 0$. Then
\[\int_{E_h}g_h\,ds\to \int_Ig\,ds.\]
\end{lemma}
\begin{proof} Possibly passing to a subsequence, we may assume that the characteristic functions of $E_h$ converge to 1 a.e. in $I$. Fatou's lemma then yields \[\displaystyle\liminf_h\int_{E_h}g_h\,ds\ge\int_Ig\,ds.\] Moreover, since  $g_h\ge 0$ and $E_h\subset I$ for every $h$, we have that $\displaystyle\int_{E_h}g_h\,ds\le \int_Ig_h\,ds$ and Fatou's lemma gives
\[\limsup_h\int_{E_h}g_h\,ds\le \limsup_h\int_{I}g_h\,ds\le\int_Ig\,ds,\]
which concludes the proof.
\end{proof}
\begin{proof}[Proof of Theorem~\ref{thm:Lav1}]
{\em Proof of Claim 1.}\\
 The first  steps of the proof follow  the path of the proof of \cite[Theorem 2.4]{ASC}, which are recalled and adapted to the more general Lagrangian considered here.
\begin{itemize}[leftmargin=*]
\item[{\em i)}] It follows from Assumptions {(${\rm B}_{y,\Lambda}$)} and  {(${\rm B}_{y,\Psi}$)}  that there are $M, {\nu_0}>0$ and a neighbourhood $\mathcal O_y$ of $y(I)$ such that:
   \begin{equation}\label{tag:M}\L(s,z,u) \le M, \quad \Psi(s,z)\le M\text{ whenever } z\in \mathcal O_y, |u|\le {\nu_0}.\end{equation}
   \item[{\em ii)}] For every $h\in\mathbb N$ there are a Lipschitz function $z_h:I\to \R^n$ and an open subset $A_h$ of $I$ such that (see Figure~\ref{fig:dimalberti}):
       \begin{itemize}
   \item[$\bullet$] $z_h(t)=y(t)$, $z_h(T)=y(T)$,
   \item[$\bullet$]  $z_h=y,\, z_h'=y'$ in $I\setminus A_h$,
   \item[$\bullet$]  $z_h$ is affine in each connected component of $A_h$,
   \item[$\bullet$]  $|A_h|\le \dfrac{T-t}{2(h+1)}$.
 \end{itemize}
 \begin{figure}[h!]
\begin{center}
\includegraphics[width=0.55\textwidth]{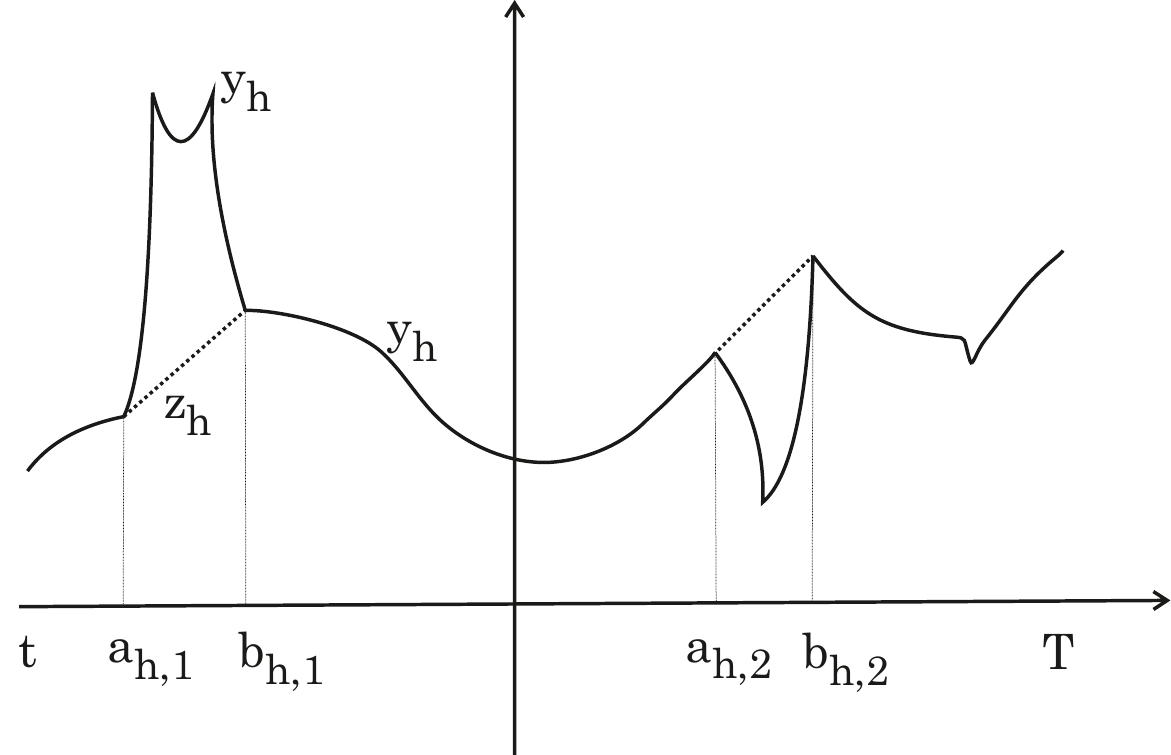}
\caption{{\small The functions $y_h$ and $z_h$.}}
\label{fig:dimalberti}
\end{center}
\end{figure}
 \item[{\em iii)}]
 Writing each $A_h$ as a countable union of open intervals $I_{h,k}:=]a_{h,k}, b_{h,k}[$, $k\in {J_h}\subseteq\N$, set
     \[\alpha_{h,k}:=|I_{h,k}|=b_{h,k}-a_{h,k},\quad \beta_{h,k}:=z_h(b_{h,k})-z_h(a_{h,k})=y(b_{h,k})-y(a_{h,k}).\]
     Then
     \[\sum_{k\in{J_h}}|\beta_{h,k}|\le \sum_{k\in\Lambda_k}\int_{I_{h,k}}|y'|\,ds=\int_{A_h}|y'|\,ds\to 0 \quad h\to \infty.\]
 \item[{\em iv)}] For every $h\in\mathbb N$ define $\varphi_h\in\ W^{1,1}(I)$ by
 \[\varphi_h(t)=t,\quad \varphi_h'=
 \begin{cases}1&\text{ in }I\setminus A_h,\\
\dfrac{|z_h'|}{\nu_0} \vee 1&\text{ in } A_h.\end{cases}\]
 \item[{\em v)}] It turns out that $\varphi_h:[t, T]$ is injective, $I\subseteq \varphi_h(I)$ , $(\varphi_h(s))_h$ converges uniformly to $Id(s):=s$ on $I$ and $|\varphi_h(I)|\to |I|$.
 \item[{\em vi)}] We have
  \[\begin{aligned}\label{tag:zhestimate}|\varphi_h(A_h)|&=\int_{\varphi_h(A_h)}1\,ds=\int_{A_h}|\varphi_h'(\tau)|\,d\tau\\
     &\le \dfrac1{\nu_0}\int_{A_h}|z_h'(s)|\,ds+|A_h|\\
     &\le \dfrac1{\nu_0}\sum_{k\in J_k}\int_{I_{h,k}}|z_h'(s)|\,ds+|A_h|\\
     &\le \dfrac1{\nu_0}\sum_{k\in J_k}|\beta_{h,k}|+|A_h|\to 0.
     \end{aligned}\]

 \item[{\em vii)}] Let $\psi_h:[t, T]\to [t, t_h]$ ($t_h\le T$) be the inverse of $\varphi_h$, restricted to $I$. Define $y_h:=z_h(\psi_h)$. Then $y_h$ is Lipschitz, $y_h(t)=y(t)$ and
     \[y_h'=z_h'(\psi_h)\psi_h'=
     \begin{cases}y'(\psi_h)&\text{ in }I\setminus \varphi_h(A_h),\\
 \left({\nu_0}\dfrac{z_h'(\psi_h)}{|z_h'(\psi_h)|}\right)\wedge z_h'(\psi_h) &\text{ in }I\cap \varphi_h( A_h).\end{cases}\]
 Moreover, $(y_h)_h$ converges to $y$ in $W^{1,p}(I, \R^n)$. We may thus assume that $y_h(I)\subset \mathcal O_y$ for every $h$.
\item[{\em viii)}]
It remains to show that $(F(y_h))_h$ converges to $F(y)$ as $h\to +\infty$; this is where the proof differs from the one of the autonomous case. We write
\[F(y_h)=\underbrace{\int_{I\setminus\varphi_h(A_h)}\L(s, y_h, y_h')\Psi(s, y_h)\,ds}_{P_{1,h}} + \underbrace{ \int_{I\cap\varphi_h(A_h)}\L(s, y_h, y_h')\Psi(s, y_h)\,ds}_{P_{2,h}}. \]
\begin{itemize}[leftmargin=*]
\item[{\em a)}] {\em  Study of the convergence of $(P_{1,h})_h$.}
From the definition of $y_h$, recalling that $z_h=y$ and $\varphi_h'=1$ out of $A_h$, we get
\[P_{1,h}=\int_{I\setminus\varphi_h(A_h)}\L(s, y(\psi_h), y'(\psi_h))\Psi(s, y(\psi_h))\,ds.\]
The change of variable $\tau=\psi_h(s)$ and the fact that $\psi_h'=1$ on $I\setminus\varphi_h(A_h)$ give
\[P_{1,h}=\int_{\psi_h(I)\setminus A_h}\L(\varphi_h(\tau), y(\tau), y'(\tau))\Psi(\varphi_h(\tau), y(\tau))\,d\tau.\]
Almost everywhere in $I$ we have (omitting the variable $\tau$ in $y(\tau), y'(\tau)$)
\begin{equation}\label{tag:P1}\L(\varphi_h, y, y')=\big[\L(\varphi_h, y, y')-\L(\tau, y, y')\big]+\L(\tau, y, y').\end{equation}
Choose $h$ big enough in such a way that $\|\varphi_h-Id\|_{\infty}<\varepsilon_*$.
Condition (S) (with $K:=\|y\|_{\infty}$) implies that,  a.e. in $ I$,
\[|\L(\varphi_h, y, y')-\L(\tau, y, y')|\le \big(\kappa \L(\tau,y,y')+\beta|y'|^p+\gamma(\tau)\big)\,\|\varphi_h-Id\|_{\infty}.\]
Therefore, since $\L(\tau, y, y')\in L^1(I)$, from \eqref{tag:M}, for all $h\in\mathbb N$ big enough,
\begin{equation}\label{tag:P5} \int_{I\setminus\varphi_h(A_h)}\!\!\!\!\!|\L(\varphi_h, y, y')-\L(\tau, y, y')|\Psi(\varphi_h, y)\,d\tau\le C_y \|\varphi_h-Id\|_{\infty},\end{equation}
for a suitable constant $C_y$, possibly depending on $y$. It follows from \eqref{tag:P1}, \eqref{tag:P5} and \eqref{tag:M} that
\begin{equation}\label{tag:P2}P_{1,h}=\int_{\psi_h(I)\setminus A_h}\L(\tau, y, y')\Psi(\varphi_h, y)\,d\tau+ \varepsilon_h, \quad \varepsilon_h\to  0. \end{equation}
Now, since $\L(\tau, y,y')\in L^1(I)$, we deduce from \eqref{tag:M} that the functions \[g_h:=\L(\tau, y, y')\Psi(\varphi_h, y)\chi_{\psi_h(I)\setminus A_h}\] are in $L^1(I)$ and, in view of (${\rm C}_{y,\Psi}$), converge a.e. in $I$ to $\L(\tau, y, y')\Psi(\tau, y)$;
it follows from Lemma~\ref{lemma:2.6} that
\[\lim_{h\to \infty}P_{1,h}=\int_I\L(\tau, y(\tau), y'(\tau))\Psi(\tau, y(\tau))\,d\tau.\]
\item[{\em b)}] {\em Study of the convergence of $(P_{2,h})_h$.}
We know that for all $h$, $y_h(I)\subset\mathcal O_y$ and $|y'_h|\le  {\nu_0}$ a.e. in $\varphi_h(A_h)\cap I$. It follows from \eqref{tag:M} that
\[\text{For a.e. }\tau\in \varphi_h(A_h)\quad \L(\tau, y_h(\tau), y_h'(\tau))\Psi(\tau, y_h(\tau))\le M^2.\]
Therefore Step {\em vi)} implies that
\[P_{2,h}\le M^2|\varphi_h(A_h)|\to  0.\]
\end{itemize}
\end{itemize}
{\em Proof of Claim 2.}\\
We proceed as above, with some additional care on the definition of $\varphi_{h}$, since we need now  that $\displaystyle\int_t^T\varphi'_{h}(s)\,ds=T-t$. Referring to the proof of Claim 1:
\begin{itemize}
\item We proceed as in Steps {\em i), ii), iii)}.
\item[] {\em iv$'$)} We modify the above Step {\em iv)} as follows.
\begin{itemize}[leftmargin=*]
\item[$\bullet$] Since $y^{-1}(U_y)$ is open in $I$ and non empty, we may choose $\overline h$ in such a way that
\begin{equation}\label{tag:estiminvU}|y^{-1}(U_y)|> 10 |A_h|\qquad h\ge \overline h.\end{equation}
Since $y$ coincides with $z_{\overline h}$ out of $A_{\overline h}$ there is $\ell>0$ satisfying  $\|y'\|_{\infty}\le \ell$ on $I\setminus A_{\overline h}$.
\item[$\bullet$]
Notice that
\[\int_{A_h}\left[\left(\dfrac{|z_h'(s)|}{\nu_0} \vee 1\right)-1\right]\, ds\to 0.\]
Indeed,  it follows from \eqref{tag:zhestimate} that
\[\int_{A_h}\left[\left(\dfrac{|z_h'(s)|}{\nu_0} \vee 1\right)-1\right]\, ds
\le |\varphi_h(A_h)|\to 0.
\]
It follows from \eqref{tag:estiminvU} that,  for $h$ big enough, there is \[\Sigma_h\subset y^{-1}(U_y)\setminus (A_{h}\cup A_{\overline h})\] of sufficiently small measure, more precisely satisfying
    \[\dfrac{|\Sigma_h|}{2}=\int_{A_h}\left[\left(\dfrac{|z_h'(s)|}{\nu_0} \vee 1\right)-1\right]\, ds.\]
    \item[$\bullet$] For each $h$ big enough, define $\varphi_h$ as follows:
 \[\varphi_h(t)=t,\quad \varphi_h'=
 \begin{cases}1&\text{ in }I\setminus (A_h\cup \Sigma_h),\\
\dfrac{|z_h'|}{\nu_0} \vee 1&\text{ in } A_h,\\
1/2&\text{ in } \Sigma_h.
\end{cases}\]
\end{itemize}
\item[] {\em v$'$)} From {\em iv$'$)}  we now have
\[\begin{aligned}\int_t^T\varphi_h'(s)\,ds&=|I\setminus (A_h\cup \Sigma_h)|+\int_{A_h}\left(\dfrac{|z_h'(s)|}{\nu_0} \vee 1\right)\, ds+\dfrac{|\Sigma_h|}{2}\\
&=|I|-|A_h|-|\Sigma_h|+\left(\dfrac{|\Sigma_h|}{2}+|A_h|\right)+\dfrac{|\Sigma_h|}{2}=|I|.
\end{aligned}\]
Thus $\varphi_h:I\to I$ is bijective.
\item[{\em vii$'$)}] Let $\psi_h$ be the inverse of $\varphi_h$. Define $y_h:=z_h(\psi_h)$. Then $y_h$ is Lipschitz, $y_h(t)=y(t)$ and
     \[y_h'=z_h'(\psi_h)\psi_h'=
     \begin{cases}y'(\psi_h)&\text{ in }I\setminus \varphi_h(A_h\cup\Sigma_h),\\
 \left({\nu_0}\dfrac{z_h'(\psi_h)}{|z_h'(\psi_h)|}\right)\wedge z_h'(\psi_h) &\text{ in }\varphi_h( A_h),
 \\2z_h'(\psi_h)=2y'(\psi_h)&\text{ in }\varphi_h(\Sigma_h).
 \end{cases}\]
 As in Claim 1, it turns out that $(y_h)_h$ converges to $y$ in $W^{1,p}(I, \R^n)$. Indeed, it follows from the proof of Claim 1 that
 \[\int_{I\setminus \varphi_h(\Sigma_h)}|y_h'-y'|^p\,ds\to 0.\]
 It remains to prove that  $\|y_h'-y'\|_{L^p(\varphi_h(\Sigma_{h}))}\to 0$. Since $|y'|\le \ell $ on $\Sigma_h$ we have
 \[\begin{aligned}\int_{\varphi(\Sigma_h)}|y_h'-y'|^p\,ds&\le 2^p\int_{\varphi(\Sigma_h)}|y_h'|^p\, ds+2^p\int_{\varphi(\Sigma_h)}|y'(s)|^p\,ds\\
 &\le 2^p\int_{\varphi(\Sigma_h)}2^p|y'(\psi_h)|^p\, ds+2^p\int_{\varphi(\Sigma_h)}|y'(s)|^p\,ds\\
&\le 2^{2p}\ell^p |\varphi(\Sigma_h)|+2^p\int_{\varphi(\Sigma_h)}|y'(s)|^p\,ds\to 0,
\end{aligned}\]
  since $\varphi_h$ is absolutely continuous and  $|\varphi_h(\Sigma_h)|\to 0$, proving the claim.\\
  We may thus assume, as in the proof of Claim 1,  that $y_h(I)\subset \mathcal O_y$ for every $h$ big enough.
  \item[{\em viii$'$)}]
It remains to show that $(F(y_h))_h$ converges to $F(y)$ as $h\to +\infty$. We write
\[F(y_h)=\underbrace{\int_{I\setminus\varphi_h(A_h\cup \Sigma_h)}*\,ds}_{P_{1,h}} + \underbrace{ \int_{\varphi_h(A_h)}*\,ds}_{P_{2,h}}+ \underbrace{ \int_{\varphi_h(\Sigma_h)}*\,ds}_{P_{3,h}}, \]
where here $*$ stands for $\L(s, y_h, y_h')\Psi(s, y_h)$. As in the proof of Claim 1 we get
\[P_{1,h}\to F(y),\quad P_{2,h}\to 0.\]
It remains to prove that $P_{3,h}\to 0$. Since $\varphi_h'\equiv \dfrac12$ on $\Sigma_h$, the change of variables $s=\varphi_{\nu}(\tau)$ gives
\[\begin{aligned}P_{3,h}&=\int_{\varphi_h(\Sigma_h)}\L(s, y_h, y_h')\Psi(s, y_h)\, ds\\
&=\int_{\varphi_h(\Sigma_h)}\L\left(s, y(\psi_h), y'(\psi_h)\psi_h'\right)\Psi(s, y(\psi_h))\,ds\\
&=\dfrac12\int_{\Sigma_h}\L\left(\varphi_h, y, 2{y'}\right)\Psi(\varphi_h, y)\,d\tau
\end{aligned}\]
Now, for a.e. $\tau$ on $\Sigma_h$, $y(\tau)\in U_y$ and $2|y'(\tau)|\le 2\ell$. Hypotheses (U$_{y, \L}$) and (B$_{y, \Psi}$) imply that $\L(\varphi_h, y, 2y')\Psi(\varphi_h, y)$ is bounded on $\Sigma_h$: the conclusion follows.
\end{itemize}
\end{proof}
\begin{remark} Notice that in the proof of Claim 1 of Theorem~\ref{thm:Lav1}, every element of the sequence $(z_h)_h$ satisfies the boundary conditions $z_h(t)=y(t), z_h(T)=y(T)$. However, since $\varphi_{h}'\ge 1$ on $[t, T]$, it may happen that $\varphi_h(T)>T$, and this  yields $\psi_h(T)<T$; as a consequence $y_h(T)=z_h(\psi_h(T))$ may differ from $z_h(T)=y(T)$.\\ In presence of a final end point condition of the form $y(T)=Y$, instead of a initial one, the proof goes as above, replacing the definition of the change of variables $\varphi_{h}$ by
 \[\varphi_h(T)=T,\quad \varphi_h'=
 \begin{cases}1&\text{ in }I\setminus A_h,\\
\dfrac{|z_h'|}{\nu_0} \vee 1&\text{ in } A_h.\end{cases}\]
\end{remark}

\section*{Acknowledgments}
I warmly thank Giovanni Alberti for the mail exchange we had during the preparation of the paper and for providing the construction of Example~\ref{ex:alberti} and for sharing the conjecture of Condition (U$_{y, \L}$) in Theorem~\ref{thm:Lav1}.
I am grateful to Giulia Treu for her comments on the manuscript  and warm encouragement.
This research is partially supported by the  Padua University grant SID 2018 ``Controllability, stabilizability and infimum gaps for control systems'', prot. BIRD 187147 and has been accomplished within the UMI Group TAA ``Approximation Theory and Applications''.
\bibliographystyle{plain}
\bibliography{Lavrentiev1dim2}
\end{document}